\documentclass[10pt,a4paper]{article}
\usepackage[utf8]{inputenc}
\usepackage{amsmath}
\usepackage{amsfonts}
\usepackage{amssymb}
\usepackage{amsthm}
\usepackage{xcolor}
\usepackage{enumitem}

\newtheorem{theorem}{Theorem}
\newtheorem{corollary}[theorem]{Corollary}

\newtheorem{lemma}[theorem]{Lemma}
\newtheorem{proposition}[theorem]{Proposition}
\newtheorem{remark}[theorem]{Remark}

\newcommand{\Pic}{\operatorname{Pic}}
\newcommand{\Proj}{\operatorname{Proj}}

\date{}

\begin{document}
%\date{}
\title{$\mathbb Z/2$-Godeaux surfaces}
\author{Eduardo Dias, Carlos Rito}
\maketitle

\begin{abstract}
We prove that the moduli space of numerical Godeaux surfaces with torsion group $\mathbb{Z}/2$ is irreducible and unirational of dimension 8. Moreover, we show that the topological fundamental group of these surfaces is also $\mathbb{Z}/2$. Our approach is based on the explicit construction of equations for all universal covers of numerical Godeaux surfaces with torsion group $\mathbb{Z}/2$.


\end{abstract}

\section{Introduction}

Let $S$ be a smooth minimal complex algebraic surface. Its topological invariants are the geometric genus $p_g,$
the irregularity $q$ and the self-intersection $K^2$ of a canonical divisor. The holomorphic Euler characteristic is
$\chi=1+p_g-q.$ Gieseker \cite{Gieseker} has shown that for each pair $(\chi,K^2)$ and $S$ of general type, 
there exists a coarse moduli space $\mathcal M_{\chi,K^2}$ that is a quasi-projective variey.
Naturally geometers want to understand which of these families are non-empty, and then if possible to classify them.
It is frustrating that this has not been achieved even for the first case in the list, the one with $\chi=K^2=1.$

For these surfaces $p_g=q=0,$ and they are known to exist since Godeaux' construction in 1931 \cite{God31}.
Nowadays surfaces of general type with $p_g=q=0, K^2=1$ are called numerical Godeaux surfaces.
Miyaoka \cite{Miy76} showed that the order of their torsion group is at most 5, and Reid \cite{Rei78} excluded the case $(\mathbb Z/2)^2$, so
their possible torsion groups are $\mathbb Z/n$ with $1\leq n\leq 5$.
Reid constructed the moduli space for the cases $n=5,4,3$,
and it follows from his work that the topological fundamental group coincides with the torsion group for $n=5,4$.
Coughlan and Urz\'ua \cite{CU18} showed that the same happens for $n=3$. In those three cases the moduli space is irreducible of dimension 8.

Coughlan \cite{Coughlan} has obtained a family of $\mathbb Z/2$-Godeaux surfaces (i.e. with torsion group $\mathbb Z/2$)
depending on 8 parameters. More recently, we have studied with Urz\'ua \cite{DRU} all possible degenerations of $\mathbb Z/2$-Godeaux
surfaces into stable surfaces with one Wahl singularity,
which produces many boundary divisors of dimension 7 in the KSBA compactification of the moduli space of these surfaces.
This is done by means of abstract constructions
(i.e. showing the existence of particular singular surfaces with no obstructions in deformations),
and computational constructions based on Coughlan's family.
We have ended up proving in \cite{DRU} that Coughlan's family is at most 7 dimensional.

For the case of Godeaux surfaces with trivial torsion,
we know the examples due to Barlow \cite{BaSC}, Craighero-Gattazzo \cite{CG} (see also \cite{RTU}),
and Lee-Park type of constructions (cf. \cite{LP}).
Catanese and LeBrun \cite{CL} proved that the Barlow
surface belongs to an irreducible component of dimension 8,
and Catanese and Pignatelli \cite{CP} proved the same for the Craighero-Gattazzo surface.

Besides the work of several other authors, these cases $n=2,1$ are still open.
Schreyer and Stenger \cite{SS} constructed an
8-dimensional family of simply connected Godeaux surfaces, but without obtaining a full classification.

Catanese and Debarre \cite{CD} showed that the \'etale double covers of $\mathbb Z/2$-Godeaux surfaces
have hyperelliptic canonical curve and birational bicanonical map onto an octic in $\mathbb P^3,$
and they did a general study of its canonical ring.
That octic is given by the determinant of a certain matrix $\alpha.$

In this paper we continue their work.
Using an idea from Miles Reid \cite{Rei90}, we get more precise information about $\alpha$
by looking first to its restriction to the case of the canonical curve, then extending to the surface.
Then we give an algorithm for the computation of all such matrices,
from which we obtain equations for the \'etale double covers of all $\mathbb Z/2$-Godeaux surfaces.
We show that their moduli space is irreducible of dimension 8, which implies that the topological fundamental group
of $\mathbb Z/2$-Godeaux surfaces is also $\mathbb Z/2.$

We note that our method is not brute force computation: for the main algorithm, the calculations used only 32 MB of RAM memory,
and took 85 seconds on a low-end computer.

Recently two special $\mathbb Z/2$-Godeaux surfaces have appeared in the literature:
a $(\mathbb Z/3)^2$-quotient of a fake projective plane, constructed by Borisov and Fatighenti \cite{BF},
which has 4 cusp singularities; a degree 6 quotient of the so-called Cartwright-Steger surface,
given by Borisov-Yeung \cite{BY}, which has 3 cusp singularities and a certain configuration of rational curves.
As an example, we give the coordinates of these surfaces in our family.

All computations are implemented with Magma \cite{BCP}, and can be found in some arXiv ancillary files.
In particular, using the files\newline \verb+3_Verifications_alpha_i_c_j.txt+ one can choose any surface in the family
and compute its invariants and singular set.

At an advanced stage of the review process, F. Catanese informed us about some overlapping with his work,
compare our Proposition \ref{PropM} with \cite{CCO}.

\subsubsection*{Acknowledgments}

The authors were supported by Portuguese Funds through FCT (Funda\c c\~ao para a Ci\^encia e a Tecnologia)
under the project PTDC/MAT-GEO/2823/2014 and CMUP (UIDB/00144/2020).

The research of the second author was partially financed by Portuguese Funds through FCT (Funda\c c\~ao para a Ci\^encia e a Tecnologia) within the Project UID/00013/2025 (https://doi.org/10.54499/UID/00013/2025).

\section{Results from Catanese-Debarre}\label{CD}

We collect here some results from the paper \cite{CD} that will be used throughout the text.

Let $S$ be the \'etale double cover of a numerical Godeaux surface with torsion group $\mathbb Z/2,$
and denote the corresponding involution by $\sigma.$
The invariants of $S$ are $K^2=2,$ $p_g=1,$ $q=0.$
Define the canonical ring of $S$ as 
$$\mathcal R=\bigoplus_{n=0}^\infty H^0(S,nK_S),$$
and let $\mathcal A=\mathbb C[x,y_1,y_2,y_3]$ be the $\mathbb C$-graded algebra with
$$\deg(x)=1,\quad \deg(y_i)=2. $$

The involution $\sigma$ acts on $\mathcal R$, the canonical ring of $S$, splitting it into eigenspaces $\mathcal R=\mathcal R^+\oplus\mathcal R^-$.
Denoting the Godeaux surface by $T$, we have 
$$\mathcal R^+\cong\bigoplus_{n\geq 0} H^0(T,nK_T),\hspace{5mm}\mathcal R^-\cong\bigoplus_{n\geq 0} H^0(T,nK_T+\eta),$$
where $\eta\in \Pic(T)$ is a $2$-torsion element. Furthermore, by Riemann-Roch,
$$\dim \mathcal R_n^+=\dim \mathcal R_n^-=1+\binom{n}{2}, \text{ for } n\geq 2.$$
Throughout the paper we denote the set of generators of $\mathcal R$ by
(see \cite[Theorem 6.1]{CD})
$$
\begin{array}{c|c}
     & x\in \mathcal R_1^- \\
    x^2, y_2\in \mathcal R_2^+ & y_1, y_3 \in \mathcal R^-_2  \\
    xy_1, xy_3, z_2, z_4\in \mathcal R_3^+ & x^3, xy_2, z_1, z_3 \in \mathcal R^-_3 \\
     & t \in \mathcal R^-_4. \\
\end{array}
$$

Notice that the vector space $\mathcal R^+_4$ is $7$-dimensional and contains
$$\{x^4,x^2y_2,y_1^2,y_1y_3,y_2^2,y_3^2,xz_1,xz_3\}.$$
Then from \cite[Lemma 4.5]{CD} there is a unique relation on these generators.
By doing a change of variables on the generators of degree 3, this relation can be written as
$$Q(y_1,y_2,y_3)+\lambda x z_1,$$
where $Q$ is a quadratic form.

%For the next 5 items see Proposition 1.1 and Theorem 6.1, Proposition 4.2 and Theorem 4.3, Theorem 4.6, and the proof of Theorem 4.6 of \cite{CD}.
We have the following.

\begin{itemize}%[leftmargin=*]
  \item[(1)] The bicanonical map of $S$ is a birational morphism and its canonical curve is hyperelliptic
             (see \cite[Proposition 1.1, Theorem 6.1]{CD}).
  
  \item[(2)] $\mathcal R$ is a Cohen-Macaulay $\mathcal A$-module,
	which implies that $\mathcal R$ admits a length one free resolution of $\mathcal A$-modules that can be written as 
    $$0\longrightarrow \mathcal A(-8)\oplus\mathcal A(-5)^4\oplus\mathcal A(-4) \xrightarrow{\ \alpha\ }
    \mathcal A\oplus\mathcal A(-3)^4\oplus\mathcal A(-4)\longrightarrow \mathcal R\longrightarrow 0,$$
    where $\alpha$ is a matrix with homogeneous entries in $\mathcal A$
    (see \cite[Proposition 4.2, Theorem 4.3]{CD}).
    
  \item[(3)] The matrix $\alpha$ can be chosen symmetric of the form
  $$
  \alpha=\left(\begin{array}{c|cccc|c}
    x^2G & xq_1 & xq_2 & xq_3 & xq_4 &  Q \\ \hline
    xq_1 & a_{11} & a_{12} & a_{13} & a_{14} & x \\
    xq_2 & a_{12} & a_{22} & a_{23} & a_{24} & 0 \\
    xq_3 & a_{13} & a_{23} & a_{33} & a_{34} & 0 \\
    xq_4 & a_{14} & a_{24} & a_{34} & a_{44} & 0 \\ \hline
    Q & x & 0 & 0 & 0 & 0 \\ 
  \end{array}\right)
  $$
  where $G,$ $q_i,$ $a_{ij}$ are of degrees $3,2,1$ in $(y_0=x^2,y_1,y_2,y_3),$ respectively.
  The $3\times 3$-minors of $(a_{ij})$ are in the ideal $(x^2,Q),$ and $\det (a_{ij})$ is in $(x^2,Q^2).$
  (See \cite{CD}, proof of Theorem 4.6, case 4).)
  
  \item[(4)] The matrix $\alpha$ satisfies the following {\em rank condition}:
  
  (RC) For each cofactor $\beta_{ij}$ of $\alpha$ there exist $l_{ij}^k\in \mathcal A$ such that $$\beta_{ij}=\sum_{k=1}^6 l_{ij}^k\beta_{1k}$$
  (see \cite[Theorem 4.6]{CD}).
  
  \item[(5)] Conversely, for any matrix $\alpha$ belonging to an open subset of the set of matrices as in $(3)$ and $(4)$,
  it is possible to define a ring structure on the $\mathcal A$-module $\mathcal R$ which $\alpha$ defines.
  The surface $X=\Proj(\mathcal R)$ is the canonical model of a minimal surface $S$ with $K^2=2,$ $p_g=1,$ $q=0$ for which
  the bicanonical map is birational onto an octic in $\mathbb P^3$ with equation $\det(\alpha).$
  (See \cite{CD}, end of Section 4.)
\end{itemize}

Furthermore, from \cite[Theorem 4.3]{Cat_Comm}:
\begin{itemize}
  \item[(6)] The equations of $X$ are given by
  $$v_iv_j=\sum_{k=1}^6l_{ij}^kv_k,\ \ \ i,j=2,\ldots,6,$$ 
  $$\sum_{j=1}^6\alpha_{ij}v_j=0,\ \ \ i=1,\ldots,6,$$
  with $$v_1=1,v_2=z_1,v_3=z_3,v_4=z_2,v_5=z_4,v_6=t.$$
\end{itemize}

\section{The matrix $\alpha$}

\begin{proposition}\label{signalpha}
The matrix $\alpha$ can be written in the form
$$
\alpha=\left(\begin{array}{c|cccc|c}
    x^2G^- & xq_1^- & xq_2^- & xq_3^+ & xq_4^+ &  Q^+ \\ \hline
    xq_1^- & a_{11}^- & a_{12}^- & a_{13}^+ & a_{14}^+ & x \\
    xq_2^- & a_{12}^- & a_{22}^- & a_{23}^+ & a_{24}^+ & 0 \\
    xq_3^+ & a_{13}^+ & a_{23}^+ & a_{33}^- & a_{34}^- & 0 \\
    xq_4^+ & a_{14}^+ & a_{24}^+ & a_{34}^- & a_{44}^- & 0 \\ \hline
    Q^+ & x & 0 & 0 & 0 & 0 \\ 
\end{array}\right)
$$
with $$Q=y_1^2-y_2^2-d^2y_3^2,$$
and where the superscript signs mean $\sigma$-invariant $(+)$ or $\sigma$-anti-invariant $(-).$

\end{proposition}

\begin{proof}
The bicanonical map of $S$ sends its (hyperelliptic) canonical curve $C$ onto the plane conic $Q=0,$
which is contained in the octic surface $\det(\alpha)=0$ in $\mathbb P^3.$
Suppose that this conic is a double line. Then $C=2D+Z,$ where $Z$ is supported on a union of $(-2)$-curves.
Since an automorphism of $S$ preserves the canonical system and $p_g(S)=1$,
then $C$ is preserved by the Godeaux involution $\sigma$.
Then the divisors $D$ and $Z$ are also preserved, giving rise to curves $C'=2D'+Z'$
in the Godeaux surface $S/\sigma,$
with $Z'$ also supported on a union of $(-2)$-curves. This contradicts the fact $C'^2=1.$
Therefore $Q$ is either a smooth conic or the union of two distinct lines,
so there is a change of variables that allow us to write $Q=y_1^2-y_2^2-d^2y_3^2,$ for some constant $d$
(we note that we could consider only the two cases $d=1$ and $d=0,$
but we prefer to keep the modulus $d$ hoping to find a family of surfaces such that one can see
the smooth conic degenerating to the singular one).

Using Riemann-Roch and the local basis of $\mathcal R$, one sees that there are two $\sigma$-invariant relations of degree $5$ and two anti-invariant ones.
Since $x,z_1,z_3,t$ are anti-invariant, $z_2,z_4$ are invariant, and we are assuming $Q$ invariant,
the relations $$\alpha\cdot (1,z_1,z_3,z_2,z_4,t)^T=0$$ from Section \ref{CD} (6)
imply that the superscript signs must be as claimed.
\end{proof}

\begin{lemma}\label{LemmaDoubleCan}
The matrix $\alpha|_{x=0}$ cannot be of the type
\[
\left(\begin{array}{c|cccc|c}
0 & 0 & 0 & 0 & 0 & Q \\ \hline
0 & y_1+d y_3 & 0 & y_2 & 0 & 0 \\
0 & 0 & y_1+dy_3 & 0 & y_2 & 0 \\
0 & y_2 & 0 & y_1-dy_3 & 0 & 0 \\
0 & 0 & y_2 & 0 & y_1-dy_3 & 0\\ \hline
Q & 0 & 0 & 0 & 0 & 0 
\end{array}\right)
\]
with $Q=y_1^2-y_2^2-d^2y_3^2.$
\end{lemma}

\begin{proof}
Since $\beta_{1k}|_{x=0}=0$ for $k=1,\ldots,5,$ in this case the rank condition (RC) is
$$\left(\beta_{ij}=l_{ij}^6\beta_{16}\right)|_{x=0}.$$
We compute (see the arXiv ancillary file Lemma2.txt) the matrix
\[
\left(l_{ij}^6\right)|_{x=0}=
\left(\begin{array}{cccccc}
0 & 0 & 0 & 0 & 0 & 1 \\
0 & y_1-d y_3 & 0 & -y_2 & 0 & 0 \\
0 & 0 & y_1-dy_3 & 0 & -y_2 & 0 \\
0 & -y_2 & 0 & y_1+dy_3 & 0 & 0 \\
0 & 0 & -y_2 & 0 & y_1+dy_3 & 0\\
1 & 0 & 0 & 0 & 0 & 0 
\end{array}\right).
\]
Now it follows from Section \ref{CD} (6) that the equations of the effective canonical divisor of the corresponding surface $S$
contain the polynomial $t^2=0,$ thus giving a double curve.
As in the proof of Proposition \ref{signalpha}, this is a contradiction.

\end{proof}

Using this result (see Case 3) below), we show the following.

\begin{proposition}\label{PropM}
The matrix $\alpha|_{x=0}$ can be considered as
\begin{equation}\label{matrixMM}
\left(\begin{array}{c|cccc|c}
0 & 0 & 0 & 0 & 0 & Q \\ \hline
0 & d^2y_3 & y_1 & y_2 & 0 & 0 \\
0 & y_1 & y_3 & 0 & y_2 & 0 \\
0 & y_2 & 0 & -y_3 & y_1 & 0 \\
0 & 0 & y_2 & y_1 & -d^2y_3 & 0 \\ \hline
Q & 0 & 0 & 0 & 0 & 0 
\end{array}\right),
\end{equation}
with $Q=y_1^2-y_2^2-d^2y_3^2.$
\end{proposition}

\begin{proof}
From Proposition \ref{signalpha}, there exist constants $r_i$ such that the matrix $\alpha|_{x=0}$ can be written as
\begin{equation*}
M:=
\begin{pmatrix}
0 & 0 & 0 & 0 & 0 & Q \\
0 & m_1 & m_2 & r_1y_2 & r_2y_2 & 0 \\
0 & m_2 & m_3 & r_3y_2 & r_4y_2 & 0 \\
0 & r_1y_2 & r_3y_2 & m_4 & m_5 & 0 \\
0 & r_2y_2 & r_4y_2 & m_5 & m_6 & 0 \\
Q & 0 & 0 & 0 & 0 & 0 \\
\end{pmatrix},
\end{equation*}
with $m_i=a_iy_1+b_iy_3.$
Denote its rows, columns by $L_i,$ $C_i,$ respectively.

If $m_2\ne 0,$ then we can assume $m_3\ne 0,$ by possibly doing the elementary operations
$L_3\rightarrow L_3+\gamma L_2,$ $C_3\rightarrow C_3+\gamma C_2,$ for some constant $\gamma.$
Now operations of the type $L_2\rightarrow L_2+\beta L_3,$ $C_2\rightarrow C_2+\beta C_3$
take us to one of the cases $m_2=0,$ $m_2\propto y_1$ or $m_2\propto y_3,$ where here the notation $a\propto b$
means that $a=\tau b$ for some constant $\tau\ne 0.$

Suppose that $m_2\propto y_1.$ If $m_3\propto y_1,$ we can go to the case $m_2=0$
(by doing elementary operations over rows and columns).
If not, we can take $m_3\propto y_3.$ Analogously if $m_2\propto y_3,$ we can obtain the case $m_2=0$ or $m_3\propto y_1.$
Now by multiplying $L_3,C_3$ and $L_2,C_2$ by constants, we can assume that
$m_2=0,$ or $m_2=y_1$ and $m_3=y_3,$ or $m_2=y_3$ and $m_3=y_1.$

Let $N$ be the $4\times 4$ central matrix of $M$ and $D:=\det(N).$
One can check that the coefficient of $y_2^4$ in $D$ is $(r_1r_4-r_2r_3)^2.$
Since $D$ is a multiple of $Q^2$ (from Section \ref{CD} (3)), we must have $r_1r_4-r_2r_3\ne 0$.
Then elementary operations over the rows $L_4,L_5$ and the columns $C_4,C_5$
allow us to assume that $r_1=r_4=1$ and $r_2=r_3=0.$

Summing up, we have three possible cases:
\begin{itemize}
\item[{\bf 1)}] $m_2=y_1,$ $m_3=y_3,$ $r_1=r_4=1,$ $r_2=r_3=0;$
\item[{\bf 2)}] $m_2=y_3,$ $m_3=y_1,$ $r_1=r_4=1,$ $r_2=r_3=0;$
\item[{\bf 3)}] $m_2=0,$\ \ \ \ \ \ \ \ \ \ \ \ \ \ \ $r_1=r_4=1,$ $r_2=r_3=0.$
\end{itemize}

Notice that the rank condition (RC) implies that each cofactor $C_{ij}$ of $N$ is divisible by the quadric $Q.$
We show below that this is enough to conclude the proof.

The arXiv ancillary file Proposition3.txt contains the computation of the cofactors that appear in the following three cases.\\

\noindent {\bf Case 1)}\\
We have
$$-C_{1,3}/y_2 = a_5y_1^2+(b_5+a_6)y_1y_3-y_2^2+b_6y_3^2,$$                
$$C_{1,4}/y_2 = a_4y_1^2+(b_4+a_5)y_1y_3+b_5y_3^2,$$
$$C_{2,3}/y_2 = (a_1a_5+a_6)y_1^2+(a_1b_5+b_1a_5+b_6)y_1y_3+b_1b_5y_3^2.$$
By comparing coefficients with $Q,$ the only possibility is that the first one is equal to $Q,$ and the other two are zero.
This implies
$$a_1=0, b_1=d^2, a_4=0, b_4=-1, a_5=1, b_5=0, a_6=0, b_6=-d^2.$$

\noindent {\bf Case 2)}\\
We have
$$-C_{1,3}/y_2 = a_6y_1^2+(a_5+b_6)y_1y_3-y_2^2+b_5y_3^2,$$                
$$C_{1,4}/y_2 = a_5y_1^2 + (a_4 + b_5)y_1y_3 + b_4y_3^2,$$
$$-C_{2,3}/y_2 = a_1a_4y_1^2 + (a_1b_4 + b_1a_4 + a_5)y_1y_3 - y_2^2 + (b_1b_4 + b_5)y_3^2.$$
The only possibility is that the second one is zero, and the other two are equal to $Q.$
This implies that $d\ne 0$ and $$a_1=d^{-2}, b_1=0, a_4=d^2, b_4=0, a_5=0, b_5=-d^2, a_6=1, b_6=0.$$

Now let
$$P:={\rm Diag}\left(1,r^3,-\frac{1}{r},\frac{1}{r^3},-r,1\right)$$
with $r^4+d^2=0.$
The change of variables $$(y_1,y_2,y_3)\mapsto \left(-r^2y_3,y_2,\frac{1}{r^2}y_1\right)$$
sends the matrix (\ref{matrixMM}) above to the product $PMP^T.$\\

\noindent {\bf Case 3)}\\
In this case we have
$$-C_{12}/y_2=y_2(a_5y_1 + b_5y_3),$$
$$-C_{13}/y_2=\left(a_3a_6y_1^2 + (b_3a_6 + a_3b_6)y_1y_3 - y_2^2 + b_3b_6y_3^2\right),$$
$$-C_{24}/y_2=\left(a_1a_4y_1^2 + (b_1a_4 + a_1b_4)y_1y_3 - y_2^2 + b_1b_4y_3^2\right).$$
This implies
$a_5=b_5=b_3a_6 + a_3b_6=b_1a_4 + a_1b_4=0,$ $a_3a_6=a_1a_4=1$ and $b_3b_6=b_1b_4=-d^2.$
Then $a_1a_3\ne 0$ and we can assume $a_1=a_3=1.$
This way we obtain 4 matrices which, by changing $y_3$ to $-y_3$, reduce to
\[
M_j:=
\begin{pmatrix}
0 & 0 & 0 & 0 & 0 & Q \\
0 & y_1+d y_3 & 0 & y_2 & 0 & 0 \\
0 & 0 & y_1-(-1)^jdy_3 & 0 & y_2 & 0 \\
0 & y_2 & 0 & y_1-dy_3 & 0 & 0 \\
0 & 0 & y_2 & 0 & y_1+(-1)^jdy_3 & 0 \\
Q & 0 & 0 & 0 & 0 & 0 \\
\end{pmatrix}
\]
with $j=1,2.$

From Lemma \ref{LemmaDoubleCan}, only the matrix $M_2$ with $d\ne 0$ can correspond to a Godeaux surface.
Let
\[
R:=
\begin{pmatrix}
1 & 0 & 0 & 0 & 0 & 0 \\
0 & i & d/2 & 0 & 0 & 0 \\
0 & -i/d & 1/2 & 0 & 0 & 0 \\
0 & 0 & 0 & -i/2 & 1/d & 0 \\
0 & 0 & 0 & id/2 & 1 & 0 \\
0 & 0 & 0 & 0 & 0 & 1
\end{pmatrix},
\]
with $i^2=-1.$
The product $RM_2R^T$ shows that $M_2$ is equivalent to a matrix of the type (\ref{matrixMM}) above.

\end{proof}

\begin{theorem}\label{ThmMatrix}
There exist $c,d\in\mathbb C$ such that the matrix $\alpha$ can be written as
\vspace{.1cm}

\[
\alpha_1=
\left(\begin{array}{c|cccc|c}
x^2G^- & xq_1^- & xq_2^- & xq_3^+ & xq_4^+ & Q \\ \hline
xq_1^- & d^2y_3 & y_1 & y_2 & 0 & x \\
xq_2^- & y_1 & y_3 & cx^2 & y_2 & 0 \\
xq_3^+ & y_2 & cx^2 & -y_3 & y_1 & 0 \\
xq_4^+ & 0 & y_2 & y_1 & -d^2y_3 & 0 \\ \hline
Q & x & 0 & 0 & 0 & 0 
\end{array}\right)
\]
$$Q=y_1^2-y_2^2-d^2y_3^2,$$

or

\[
\alpha_2=
\left(\begin{array}{c|cccc|c}
x^2G^- & xq_1^- & xq_2^- & xq_3^+ & xq_4^+ & Q \\ \hline
xq_1^- & y_3 & ey_1 & y_2 & 0 & x \\
xq_2^- & ey_1 & y_1 & cx^2 & y_2 & 0 \\
xq_3^+ & y_2 & cx^2 & -y_1 & ey_1 & 0 \\
xq_4^+ & 0 & y_2 & ey_1 & -y_3 & 0 \\ \hline
Q & x & 0 & 0 & 0 & 0 
\end{array}\right)
\]
$$d\ne 0,\ e:=-\frac{1}{2d},\ Q=e^2y_1^2-y_2^2-y_1y_3,$$

or

\[
\alpha_3=
\left(\begin{array}{c|cccc|c}
x^2G^- & xq_1^- & xq_2^- & xq_3^+ & xq_4^+ & Q \\ \hline
xq_1^- & y_3 & y_1 & y_2 & 0 & x \\
xq_2^- & y_1 & 0 & cx^2 & y_2 & 0 \\
xq_3^+ & y_2 & cx^2 & 0 & y_1 & 0 \\
xq_4^+ & 0 & y_2 & y_1 & -y_3 & 0 \\ \hline
Q & x & 0 & 0 & 0 & 0 
\end{array}\right)
\]
$$d=0,\ Q=y_1^2-y_2^2,$$
\vspace{.1cm}

\noindent with
$G,$ $q_i$ polynomials of degree $3,2$ in $\left(y_0=x^2,y_1,y_2,y_3\right),$ respectively.\newline
(As above the superscript signs mean $\sigma$-invariant or $\sigma$-anti-invariant.)

Moreover, we can assume $c=1$ or $c=0.$
\end{theorem}
\begin{remark}
The cases $\alpha_2$ and $\alpha_3$ could be combined into a single case, but due to the difficulty of the
computations in Section~\ref{thealgorithm} we need to treat them separately.
\end{remark}
\begin{proof}
We want to extend the matrix (\ref{matrixMM}) from Proposition \ref{PropM} by adding polynomials divisible by $x.$
This must respect the signs given in Proposition \ref{signalpha},
hence concerning the entries of order $2,$ we can only add multiples of $x^2$ to the $\sigma$-invariant ones.
We get the matrix
\[
\alpha=
\begin{pmatrix}
x^2G^- & xq_1^- & xq_2^- & xq_3^+ & xq_4^+ & Q \\
xq_1^- & d^2y_3 & y_1 & y_2+c_1x^2 & c_2x^2 & c_5x \\
xq_2^- & y_1 & y_3 & c_3x^2 & y_2+c_4x^2 & c_6x \\
xq_3^+ & y_2+c_1x^2 & c_3x^2 & -y_3 & y_1 & 0 \\
xq_4^+ & c_2x^2 & y_2+c_4x^2 & y_1 & -d^2y_3 & 0 \\
Q & c_5x & c_6x & 0 & 0 & 0 
\end{pmatrix}.
\]

We know that $\det(\alpha)$  defines an irreducible surface in $\mathbb P^3,$ thus $c_5=c_6=0$ is impossible.

We consider 4 cases. (The computational details are available in the arXiv ancillary file Theorem4.txt.)\\

\noindent {\bf Case 0: $c_6=0$}\\
Since $c_5\ne 0$ in this case, we can take $c_5=1$ from the change of variable $x\rightarrow x/{c_5}.$
Then elementary operations using the last row and column give us $c_1=c_4$ and $c_2=0.$
We can assume $c_4=0$ by doing $y_2\rightarrow y_2-c_4x^2$, followed by an elementary operation on the first row/column
to preserve $Q$, obtaining a matrix of the type $\alpha_1$.\\

\noindent {\bf Case 1: $c_6\ne 0$, $d^2\ne (c_5/c_6)^2,$ $c_5\ne 0$}\\
Let $$a:=\frac{c_5^2+d^2c_6^2}{c_5^2-d^2c_6^2},\ \ \ \ \ \ b:=\frac{2c_5c_6}{c_5^2-d^2c_6^2}$$
and
\[
P:=
\begin{pmatrix}
1 & 0 & 0 & 0 & 0 & 0 \\
0 & r & -d^2rc_6/c_5 & 0 & 0 & 0 \\
0 & -rc_6/c_5 & r & 0 & 0 & 0 \\
0 & 0 & 0 & r & rc_6/c_5 & 0 \\
0 & 0 & 0 & d^2rc_6/c_5 & r & 0 \\
0 & 0 & 0 & 0 & 0 & 1 
\end{pmatrix},
\]
with $$r^2=\dfrac{c_5^2}{c_5^2-d^2c_6^2}.$$
The product $P\alpha P^T$ is a matrix of the type
\[
\begin{pmatrix}
x^2G' & xq_1' & xq_2' & xq_3' & xq_4' & Q \\
xq_1' & d^2Y_3 & Y_1 & y_2+c_1'x^2 & c_2'x^2 & c_5'x \\
xq_2' & Y_1 & Y_3 & c_3'x^2 & y_2+c_4'x^2 & 0 \\
xq_3' & y_2+c_1'x^2 & c_3'x^2 & -Y_3 & Y_1 & 0 \\
xq_4' & c_2'x^2 & y_2+c_4'x^2 & Y_1 & -d^2Y_3 & 0 \\
Q & c_5'x & 0 & 0 & 0 & 0 
\end{pmatrix},
\]
with
\[
\begin{pmatrix}
    Y_1      \\
    Y_3      
\end{pmatrix}
=
\begin{pmatrix}
    a  &  -d^2b      \\
    -b  &  a      
\end{pmatrix}
\begin{pmatrix}
    y_1      \\
    y_3      
\end{pmatrix} .
\]
Notice that the determinant of this $2\times 2$ matrix is $a^2-d^2b^2=1.$

Since $c_5'=c_5/r\ne 0,$ we can proceed as before to get $c_5'=1,c_1'=c_2'=c_4'=0.$
Finally from $$Y_1^2-y_2^2-d^2Y_3^2 = y_1^2-y_2^2-d^2y_3^2$$ we see that the matrix $P\alpha P^T$
is in the form of the matrix $\alpha_1$ above.
\\

\noindent {\bf Case 2: $c_6\ne 0$, $d^2\ne (c_5/c_6)^2,$ $c_5=0$}\\
Let
\[
P:=
\begin{pmatrix}
1 & 0 & 0 & 0 & 0 & 0 \\
0 & 0 & d & 0 & 0 & 0 \\
0 & 1/d & 0 & 0 & 0 & 0 \\
0 & 0 & 0 & 0 & 1/d & 0 \\
0 & 0 & 0 & d & 0 & 0 \\
0 & 0 & 0 & 0 & 0 & 1 
\end{pmatrix}.
\]
The product $P\alpha P^T$ gives us a matrix of the type $\alpha$ with $\alpha_{3,6}=\alpha_{6,3}=0.$
We proceed as in Case 0 to get $\alpha_1.$
\\

\noindent {\bf Case 3: $c_6\ne 0$, $d^2=(c_5/c_6)^2$}\\
Since $c_6\ne 0,$ we can take $c_6=1,$ $c_5=\pm d$.
We now proceed as in Case 0: elementary operations using the last row and column give $c_3=c_4=0,$
then a change of variable involving $y_2$ and $x$ gives $c_1=0$
(followed by an elementary operation on the first row/column to preserve $Q$).

By looking to $P\alpha P^T$ with $P:={\rm Diag}(1,-1,1,-1,1,1),$ we see that we can consider $c_5=d.$

Let
\[
P':=
\begin{pmatrix}
1 & 0 & 0 & 0 & 0 & 0 \\
0 & 0 & 1 & 0 & 0 & 0 \\
0 & 1 & -d & 0 & 0 & 0 \\
0 & 0 & 0 & d & 1 & 0 \\
0 & 0 & 0 & 1 & 0 & 0 \\
0 & 0 & 0 & 0 & 0 & 1 
\end{pmatrix}.
\]
If $d=0$, the product $P'\alpha P'^T$ is a matrix of the type $\alpha_3$.
If $d\ne 0$, the product $P'\alpha P'^T$, followed by the change $y_1\rightarrow ey_1+dy_3$, is a matrix of the type $\alpha_2$.\\

\noindent{\bf Finally, the assertion about $c$}:
Assuming $c\ne 0,$ take $s$ such that $c=s^4.$ We can assume $c=1$ by taking the product
$P\alpha P^T$ with $$P:={\rm Diag}(1,s,1/s,1/s,s,1)$$
followed by the changes $$x\mapsto x/s,\, y_3\mapsto s^2y_3,\, d\mapsto d/(s^2)\ \ \ \ {\rm for\ the\ matrix\ } \alpha_1,$$
$$x\mapsto x/s,\, y_1\mapsto s^2y_1,\, y_3\mapsto y_3/s^2,\, e\mapsto e/s^2\ \ \ \ {\rm for\ the\ matrix\ } \alpha_2,$$
and
$$x\mapsto x/s,\, y_3\mapsto y_3/s^2\ \ \ \ {\rm for\ the\ matrix\ } \alpha_3.$$

\end{proof}

\section{The equations}\label{algo}

Recall from Theorem \ref{ThmMatrix} that $G$ is a degree 3 anti-invariant polynomial on the variables $(x^2,y_1,y_2,y_3),$
$q_1,q_2$ are anti-invariant of degree 2, and $q_3,q_4$ are invariant of degree 2.
We write
\begin{equation*}
\begin{split}
G=g_1x^4y_1+g_2x^4y_3+g_3x^2y_1y_2+g_4x^2y_2y_3+g_5y_1^3+\\
g_6y_1^2y_3+g_7y_1y_2^2+g_8y_1y_3^2+g_9y_2^2y_3+g_{10}y_3^3,
\end{split}
\end{equation*}
\begin{align*}
q_1&=b_1x^2y_1+b_2x^2y_3+b_3y_1y_2+b_4y_2y_3,\\
q_2&=b_5x^2y_1+b_6x^2y_3+b_7y_1y_2+b_8y_2y_3,\\
q_3&=b_{9}x^4+b_{10}x^2y_2+b_{11}y_1^2+b_{12}y_1y_3+b_{13}y_2^2+b_{14}y_3^2,\\
q_4&=b_{15}x^4+b_{16}x^2y_2+b_{17}y_1^2+b_{18}y_1y_3+b_{19}y_2^2+b_{20}y_3^2.
\end{align*}

Our goal is to compute the set of parameters $g_1,\ldots,g_{10}$ and $b_1,\ldots,b_{20}$ such that the matrices $\alpha_i$
from Theorem \ref{ThmMatrix} satisfy the rank condition (RC).
We note that the equalities $\beta_{ij}=\sum_{k=1}^6 l_{ij}^k\beta_{1k}$ hold in the polynomial ring $\mathcal A,$
so comparison of coefficients is valid here.

By doing elementary operations over the rows and columns of the matrix $\alpha_i,$ we can assume that $8$ of the $b_i$ are zero.
For instance in the case of $\alpha_1:$
\begin{itemize}
\item We remove the monomials $x^2y_2, y_1^2, y_2^2$ from $q_4$\newline
     (using the lines/columns $l_3,l_4/c_3,c_4$);
\item We remove the monomials $y_1^2,y_1y_3,y_3^2$ from $q_3$\newline
     (using the lines/columns $l_4,l_5/c_4,c_5$);
\item We remove the monomials $x^2y_1, x^2y_3$ from $q_1$\newline
     (using the line/column $l_6/c_6$).
\end{itemize}

The idea for the computations is the following:
we write the polynomials $l_{ij}^k$ as a sum with coefficients $r_m$
and monomials in $\mathcal A=\mathbb C[x,y_1,y_2,y_3]$
of the right degree and eigenspace, i.e. according to the (RC)
$$\beta_{ij}-\sum_{k=1}^6 l_{ij}^k\beta_{1k}=0.$$
Then we need to determine values for the parameters $d,g_p,b_n,r_m$ such that
the coefficients of the polynomials from (RC) vanish identically.
The aim is to eliminate all parameters $r_m$ first, then to eliminate some of the remaining parameters
to obtain a space of dimension 8.
After this the equations of the surfaces $X,$ which depend on these remaining parameters, follow from Section \ref{CD} (6).

The polynomials $l_{ij}^k$ depend on 371 parameters,
and the cofactors $\beta_{i,j}$ depend on $d,g_1,\ldots,g_{10}$ and twelve of $b_1,\ldots,b_{20}.$
We have a system of 876 coefficients depending on $371+23=394$ parameters.
All parameters $r_m$ appear with degree 1, some of them isolated.
A basic linear elimination applied to each of the cases of Theorem \ref{ThmMatrix} is enough to finish the computation.

\subsection{The computation}\label{thealgorithm}

Given a scheme $S$, the Magma function {\bf LinearElimination$(S,v)$} iteratively eliminate variables that appear strictly linearly on the defining equations of $S$, following the variable order given by $v$. We use this function below.\\

%We can now start.

\begin{itemize}

\item[(1)] We work on \(R[x,y_1,y_2,y_3,z_1,z_2,z_3,z_4,t]\)
with \(R\) a polynomial ring of rank \(394\) (the parameters).
Recall that the involution \(\sigma\) is
$$(x,y_1,y_2,y_3,z_1,z_2,z_3,z_4,t)\mapsto (-x,-y_1,y_2,-y_3,-z_1,z_2,-z_3,z_4,-t).$$

\item[(2)] We define the matrix \(\alpha_j,\) with $G, q_i$ depending on the parameters $g_m,b_n.$

\item[(3)] We define the polynomials \(l_{ij}^k\) depending on parameters $r_p.$
Notice that these must be chosen with the right degree and \(\sigma\) sign.

\item[(4)] We write the polynomials that define the rank condition and a sequence $f$
containing their coefficients. Our goal is to determine relations on the $g_m,b_n,d$
such that those polynomial coefficients vanish.

\item[(5)] Using the matrix \(\alpha_j\) and \(l_{ij}^k,\) we compute the defining polynomials $F$ given by Section 2 (6),
on the variables \(x,y_1,y_2,y_3,z_1,z_2,z_3,z_4,t\).

\item[(6)] Now we compute $${\rm LinearElimination} (f\ {\rm cat}\ F,v),$$
 with
$$v:=\left(r_1,\ldots,r_{371},g_1,\ldots,g_{10},b_1,\ldots,b_{12},d\right).$$
(Here cat denotes concatenation.)

This eliminates all polynomials from the sequence $f$ and outputs the polynomials
$F$ now depending on (at most) 9 of the parameters $d,g_i,b_j$ and some of the $r_m.$

\item[(7)] 
The equations $F_1=\cdots=F_5=0$ of degree $\leq 5$ do not depend on the $r_m,$
and we check that, on the equations of degree $>5,$
the $r_m$ appear as $r_mG_m$ with $G_m$ in the ideal $\langle F_1,\ldots,F_5\rangle .$
The $F_j=0$ can then be used to eliminate all the $r_m,$ thus we can take $r_m=0.$
This gives the final equations depending on (at most) 9 of the parameters $d,g_i,b_j$.\\

\end{itemize}

The arXiv ancillary file \verb+1_TheAlgorithm_alpha_1_c_1.pdf+
contains a Magma implementation of this algorithm for the case $\alpha_1$ with $c=1$ of Theorem \ref{ThmMatrix},
which gives a 9-parameter family of surfaces of general type with invariants
$p_g=1, q=0$ and $K^2=2$ (with a fixed point free involution).

The computations for the other cases $\alpha_i,c=j$ are given in the files\newline \verb+1_The_Algorithm_alpha_i_c_j.txt+.
%We don't need to consider the case $\alpha_3$ with $c=0$ because here the determinant of the matrix $\alpha_3$ is
%a square, thus does not correspond to an irreducible surface.

\section{The moduli space}

Denote by $\mathcal M$ the moduli space of numerical Godeaux surfaces with torsion group $\mathbb Z/2,$
and let $\mathcal M_i^j$ be the subset of $\mathcal M$ corresponding to the matrix $\alpha_i$ with $c=j$
of Theorem $\ref{ThmMatrix}.$

\begin{theorem}\label{SpacesM}
The space $\mathcal M_1^1$ is dominated by an open dense subset of the $8$-dimensional weighted projective space $\mathbb P(1,1,2,2,2,3,3,4,4).$

We have $\mathcal M_2^0=\mathcal M_3^0=\emptyset.$
The spaces $\mathcal M_1^0,$ $\mathcal M_2^1$ and $\mathcal M_3^1$ are at most $7$-dimensional, and are contained in the closure of $\mathcal M_1^1.$

\end{theorem}

\begin{proof}
The computations of Section \ref{thealgorithm} give a matrix $\alpha_1, c=1$ whose entries are polynomials on the variables
$y_0=x^2,y_1,y_2,y_3$ with coefficients depending on $9$ parameters $b_5,b_9,b_6,b_8,d,b_2,b_{11},g_9,b_{12}.$
Each determinant $D:=\det(\alpha_1)$ gives an octic surface in $\mathbb P^3$
which, for general values of the parameters, is the image of $Y$ by its bicanonical map,
where $Y$ is the surface that has also been computed by the algorithm.
We check that, for any nonzero constant $u,$ we have
$$D=D(y_0/u,y_1,y_2,y_3/u,ub_5,ub_9,u^2b_6,u^2b_8,u^2d,u^3b_2,u^3b_{11},u^4g_9,u^4b_{12})$$
(see the arXiv ancillary file \verb+4_ItIsWeightedProjSpace_alpha_1_c_1.txt+).\newline
Therefore, the octics corresponding to $(b_5,b_9,b_6,b_8,d,b_2,b_{11},g_9,b_{12})$ and
$$(ub_5,ub_9,u^2b_6,u^2b_8,u^2d,u^3b_2,u^3b_{11},u^4g_9,u^4b_{12})$$ are identified by the change of variables
$(y_0/u,y_1,y_2,y_3/u).$ This implies that the above family of octic surfaces is parametrized by
$\mathbb P(1,1,2,2,2,3,3,4,4).$

For $\mathcal M_3^1$ the equations depend only on 8 parameters,
and we have a similar action giving the structure of a projective space, hence
$\mathcal M_3^1$ is of dimension at most 7.
In the case of $\mathcal M_1^0$ the equations depend on 9 parameters, but we have the following 2-dimensional action:
$$D=D(y_0/k,y_1,y_2,y_3/u,kb_5,kb_9,kub_6,k^2b_8,u^2d,ku^2b_2,ku^2b_{11},k^2u^2g_9,ku^3b_{12}).$$
This implies that $\mathcal M_1^0$ is of dimension at most 7.
(See the corresponding arXiv ancillary files.)

We show that the space $\mathcal M_2^1$ is also projective, $\mathbb P(1,1,1,2,2,2,3,3,4)$
with variables $e,b_4,b_9,b_5,b_8,b_{12},b_1,b_7,g_7$ $(e=-1/(2d)).$
Since $e\ne 0,$ we can consider $e=1.$
The corresponding family of octic surfaces in $\mathbb P^3$
is given by a single polynomial $$D_{b_4,b_9,b_5,b_8,b_{12},b_1,b_7,g_7}(y_0,y_1,y_2,y_3),$$ of degree 8 on the $y_i.$
We want to know if there are two such octics that are projectively equivalent,
i.e. if there are parameters $b_4',b_9',b_5',b_8',b_{12}',b_1',b_7',g_7'$ and a change of variables $Y_i$ such that the
coefficients of the polynomial equation
\begin{multline}\label{polycf}
F(y_0,y_1,y_2,y_3):=D_{b_4,b_9,b_5,b_8,b_{12},b_1,b_7,g_7}(Y_0,Y_1,Y_2,Y_3)
\\
-D_{b_4',b_9',b_5',b_8',b_{12}',b_1',b_7',g_7'}(y_0,y_1,y_2,y_3)=0
\end{multline}
vanish. From computer experiments we were able to find a 1-dimensional change of variables $Y_i:=y_i(k)$ and polynomials
$b_4',b_9',b_5',b_8',b_{12}',b_1',b_7',g_7'$
on the parameters $k,b_4,b_9,b_5,b_8,b_{12},b_1,b_7,g_7$ such that equation (\ref{polycf}) holds (see the arxiv ancillary file \verb+6_Dimension_alpha_2_c_1.txt+).
This implies that the dimension of $\mathcal M_2^1$ is at most 7.

It is easy to check that for $c=0$ the determinant of the matrix $\alpha_3$ is a square, hence $\mathcal M_3^0=\emptyset.$
The computations for the case $\alpha_2,c=0$ give equations that contain the point $(0:0:0:1:0:0:0:0:0).$
Since this would be a base point for the tricanonical map,
this is not possible for surfaces with $p_g=1,$ $q=0,$ $K^2=2.$ Thus $\mathcal M_2^0=\emptyset.$

By the results of Kuranishi \cite{Kuranishi} and Wavrik \cite{Wavrik} (as explained in \cite{Cat_Moduli}),
the number of moduli of each $\mathbb Z/2$-Godeaux surface is at least 8.
Since all the $\mathcal M_i^j$ are of dimension $\leq 7$ except for $i=j=1$,
then the dimension of $\mathcal M_1^1$ is 8 and
the spaces $\mathcal M_1^0,$ $\mathcal M_2^1,$ $\mathcal M_3^1$ must be contained in the closure of $\mathcal M_1^1.$

\end{proof}

\begin{corollary}
The moduli space of numerical Godeaux surfaces with torsion group $\mathbb Z/2$ is irreducible and unirational of dimension 8.
The topological fundamental group of these surfaces is also $\mathbb Z/2.$
\end{corollary}

\begin{proof}
The first part is immediate from Theorem \ref{SpacesM}.
For the second part it suffices to note that there exist $\mathbb Z/2$-Godeaux surfaces with topological fundamental group $\mathbb Z/2,$
see \cite{Ba}.
\end{proof}

\section{Two special surfaces}

Borisov and Fatighenti \cite{BF} give the equations of a surface $X$ with an action of $\mathbb Z/3$ such that
the surface $Y:=X/(\mathbb Z/3)$ is an \'etale double covering of a $\mathbb Z/2$-Godeaux surface with 4 cusps,
which in turn is a $(\mathbb Z/3)^2$-quotient of a fake projective plane.

The surface $X$ is embedded in $\mathbb P^7$ by its bicanonical map,
and the action of $\mathbb Z/3$ is $$(x_0:x_1:x_2:x_3:x_4:x_5:x_6:x_7)\mapsto (x_0:x_2:x_3:x_1:x_5:x_6:x_4:x_7).$$
The map given by $(x_0:x_1+x_2+x_3:x_4+x_5+x_6:x_7)$ sends $X$ to the bicanonical image of $Y,$
an octic surface in $\mathbb P^3.$ Magma gives the equation of this surface, and with some computations
we get its coordinates in our family (case $\alpha_1,$ $c=1$):
$$(b_5,b_9,b_6,b_8,d,b_2,b_{11},g_9,b_{12})=$$
$$( 36r+36,64,-360r+1752,360r+4392,-30r-366,-10176r-45504,$$ $$20976r+78960,238008r+1635576,-383328r+867744 ),$$
with $r=\sqrt{-15}.$

Now with computations analogous to the ones in the file\newline \verb+3_Verifications_alpha_1_c_1.txt+, one can see that the surface $Y$ is as claimed.\\

Borisov and Yeung \cite{BY} give the equations of a $\mathbb Z/3$-quotient $Z$ of the Cartwright-Steger surface,
and they show that $Z$ is an \'etale double covering of a $\mathbb Z/2$-Godeaux surface. The surface $Z$ has 6 cusp singularities and
contains 3 disjoint $(-3)$-curves. Proceeding as above,
we find the image of $Z$ by its bicanonical map, and then we compute its coordinates in our family (case $\alpha_1,$ $c=1$):
$$(b_5,b_9,b_6,b_8,d,b_2,b_{11},g_9,b_{12})=\left(-60, 40, -120, -302, 9, 252, 360, 15903, 648\right).$$
Again with computations analogous to the ones in the file\newline \verb+3_Verifications_alpha_1_c_1.txt+, one can check the invariants of $Z$
and its singularities.

\bibliography{References_V10}

\vspace{0.8cm}

\noindent Eduardo Dias
\vspace{0.1cm}
\\ Departamento de Matem\' atica
\\ Faculdade de Ci\^encias da Universidade do Porto
\\ Rua do Campo Alegre 687
\\ 4169-007 Porto, Portugal
\\ www.fc.up.pt, {\tt eduardo.dias@fc.up.pt}\\

\vspace{0.3cm}

\noindent Carlos Rito
\vspace{0.1cm}
\\ Centro de Matem\'atica, Universidade do Minho - Polo CMAT-UTAD
\\ Universidade de Tr\'as-os-Montes e Alto Douro, UTAD
\\ Quinta de Prados
\\ 5000-801 Vila Real, Portugal
\\ www.utad.pt, {\tt crito@utad.pt}

\end{document}